\documentclass[12pt, reqno]{amsart}
\usepackage{amsmath, amsthm, amscd, amsfonts, amssymb, graphicx, color}
\usepackage[bookmarksnumbered, colorlinks, plainpages]{hyperref}

\newtheorem{theorem}{Theorem}[section]
\newtheorem{lemma}[theorem]{Lemma}

\newtheorem{corollary}[theorem]{Corollary}
\theoremstyle{definition}

\newtheorem{example}[theorem]{Example}

\theoremstyle{remark}
\newtheorem{remark}[theorem]{Remark}
\numberwithin{equation}{section}

\begin{document}
\setcounter{page}{1}


\title [Obtaining Leibniz's rule for derivations in its most general form] {Obtaining Leibniz's rule for derivations in its most general form}
\author[Amin Hosseini]{Amin Hosseini}
\address{Amin Hosseini, Department of Mathematics, Kashmar Higher Education Institute, Kashmar, Iran}
\email{\textcolor[rgb]{0.00,0.00,0.84}{a.hosseini@mshdiau.ac.ir}}




\date{Received: xxxxxx; Revised: yyyyyy; Accepted: zzzzzz.
\newline \indent $^{*}$ Corresponding author}

\begin{abstract}
The main purpose of this paper is to obtain Leibniz's rule for generalized types of derivations via Newton's binomial formula. In fact, we provide a short formula to calculate the nth power of any kind of derivations.
\end{abstract} \maketitle


\section{  the connection between Leibniz's rule and Newton's binomial formula}
The purpose of this section is to obtain Leibniz's rule for generalized types of derivations via Newton's binomial formula. Throughout the article, $\mathcal{A}$ denotes an algebra, $I$ is the identity mapping and $\textbf{1}$ stands for the identity element of any unital algebra.
\\
\\
Let $d: \mathcal{A} \rightarrow \mathcal{A}$ be an additive derivation and let $n$ be a positive integer. According to Part (i) of \cite[Theorem 1.8.5]{D}, Leibniz's rule or Leibnz's identity is as follows:
\begin{align}
d^{n}(ab) = \sum_{k = 0}^{n}\Big(_{k}^{n}\Big)d^{n - k}(a)d^{k}(b), \ \ \ \ \ \ (a, b \in \mathcal{A}),
\end{align}
where $d^{0} = I$. Suppose $\mathcal{A}$ is unital and $x, y$ are two arbitrary elements of $\mathcal{A}$ such that $xy = yx$. According to Newton's binomial formula, we have
\begin{align}
(x + y)^{n} = \sum_{k = 0}^{n}\Big(_{k}^{n}\Big)x^{n - k}y^{k},
\end{align}
where $a^{0} = \textbf{1}$ for any $a \in \mathcal{A}$. If we look carefully at the above two formulas, we will notice the similarity between them. This similarity is the basis of our idea to derive Leibniz's rule for generalized types of derivations. First, we explain our idea of calculating Leibniz's rule for ordinary derivations. Let $a$ and $b$ be two arbitrary elements of $\mathcal{A}$. We have $d(ab) = d(a) b+ a d(b)$. Letting $\mathfrak{X} = d(a)b$ and $\mathfrak{Y} = a d(b)$, we see that
\begin{align*}
d(ab) = \mathfrak{X} + \mathfrak{Y}.
\end{align*}
Now, we define the operation $\otimes$ between $\mathfrak{X}$ and $\mathfrak{Y}$ as follows:
\begin{align*}
\mathfrak{X} \otimes \mathfrak{Y} = (d(a)b) \otimes (a d(b)) = d(a) d(b).
\end{align*}
Clearly, we have
\begin{align*}
\mathfrak{Y} \otimes \mathfrak{X} = (ad(b)) \otimes (d(a)b) = d(a) d(b).
\end{align*}
As can be seen, $\mathfrak{X} \otimes \mathfrak{Y} = \mathfrak{Y} \otimes \mathfrak{X}$. Moreover, have
\begin{align*}
& \mathfrak{X}^{2} = \mathfrak{X} \otimes \mathfrak{X} = (d(a)b) \otimes (d(a)b) = d^{2}(a) b, \\ &  \mathfrak{Y}^{2} = \mathfrak{Y} \otimes \mathfrak{Y} = (a d(b)) \otimes (a d(b)) = a d^2(b).
\end{align*}
Hence, for any positive integer $n$, we have the following expressions:
\begin{align*}
& \mathfrak{X}^{n} = d^{n}(a) b, \\ &  \mathfrak{Y}^{n} = a d^n(b).
\end{align*}
Also, we define $\mathfrak{X}^{0} = \mathfrak{Y}^{0} = ab$. Considering the above equations and the fact that $\mathfrak{X} \otimes \mathfrak{Y} = \mathfrak{Y} \otimes \mathfrak{X}$ and also using Newton binomial formula, we get that
\begin{align*}
(\mathfrak{X} + \mathfrak{Y})^{n} & = \sum_{k = 0}^{n}\Big(_{k}^{n}\Big)\mathfrak{X}^{n - k} \otimes \mathfrak{Y}^{k} \\& = \sum_{k = 0}^{n}\Big(_{k}^{n}\Big)(d^{n - k}(a)b)\otimes (a d^{k}(b)) \\ & = \sum_{k = 0}^{n}\Big(_{k}^{n}\Big)d^{n - k}(a)d^{k}(b) \\ & = d^{n}(ab),
\end{align*}
which means that
\begin{align*}
d^{n}(ab) = (\mathfrak{X} + \mathfrak{Y})^{n}.
\end{align*}
The above discussion shows the connection between Leibniz's rule and Newton's binomial formula. \\

\section{calculation of Leibniz's rule via commutative Newton's binomial formula}

In this section, we are going to calculate Leibniz's rule for generalized types of derivations via Newton's binomial formula. So far, various generalizations for derivations have been defined. One of them is $(\sigma, \tau$)-derivation. Let $\sigma, \tau:\mathcal{A} \rightarrow \mathcal{A}$ be two mappings. An additive mapping $d: \mathcal{A} \rightarrow \mathcal{A}$ is called a $(\sigma, \tau$)-derivation if $$d(ab) = d(a) \sigma(b) + \tau(a) d(b)$$ holds for all $a, b \in \mathcal{A}$. By a $\sigma$-derivation we mean a $(\sigma, \sigma$)-derivation. Clearly, if $\sigma = I$, then we reach an ordinary derivation on $\mathcal{A}$. Also, an additive mapping $\Delta: \mathcal{A} \rightarrow \mathcal{A}$ is called a generalized $(\sigma, \tau$)-derivation corresponding to an additive mapping $d: \mathcal{A} \rightarrow \mathcal{A}$ if
$$ \Delta(ab) = \Delta(a) \sigma(b) + \tau(a) d(b)$$
holds for all $a, b \in \mathcal{A}$. By a generalized $\sigma$-derivation we mean a generalized $(\sigma, \sigma$)-derivation. In particular, if $\sigma = I$, we reach a generalized derivation. Now, we consider a generalized form of derivations covering the above-mentioned notions. Let $\mathcal{A}$ be an algebra, let $f, g_1, g_2, h_1, h_2:\mathcal{A} \rightarrow \mathcal{A}$ be mappings such that $f$ is additive. Suppose these mappings satisfy the following equation:
$$ f(ab) = g_1(a) h_1(b) + g_2(a) h_2(b)$$ for all $a, b \in \mathcal{A}$. In this case, the additive mapping $f$ is called a $(g_1, h_1, g_2, h_2)$-derivation. As can be seen, this type of derivation covers the notion of derivation, generalized derivation, $(\sigma, \tau)$-derivation, generalized $(\sigma, \tau)$-derivation, ternary derivation, homomorphism, left and right centralizer. Our main goal in this section is to obtain Leibniz's rule for this generalized type of derivations.
\\
\\
Let $T$ and $S$ be two mappings from $\mathcal{A}$ into itself and let $a$ be an arbitrary element of $\mathcal{A}$. We define an operation $\otimes$ between $ T(a)$ and $ S(a)$ as follows: $$  T(a) \otimes S(a)=  TS(a).$$
It is clear that if $\alpha$ and $\beta$ are two arbitrary complex numbers, then
$$ \left(\alpha T(a)\right) \otimes \left(\beta S(a)\right) = \alpha\beta TS(a).$$
Let $m$ be a positive integer and let $T_1, ..., T_m$ and $S_1, ..., S_m$ be mappings of an algebra $\mathcal{A}$. Let $a_1, ..., a_m$ be arbitrary elements of $\mathcal{A}$. We define the operation $\otimes$ between $T_1(a_1) ... T_m(a_m)$ and $S_1(a_1) ... S_m(a_m)$ as follows:
\begin{align}
(T_1(a_1) ... T_m(a_m)) \otimes (S_1(a_1) ... S_m(a_m)) = \prod_{i = i}^{m}T_i S_i(a_i)
\end{align}
Let $a,b$ be two arbitrary elements of $\mathcal{A}$ and let $\mathfrak{X}_i = T_i(a)S_i(b)$, where $i \in \{1, 2, 3\}$. We assume that
\begin{align}
& \left(\mathfrak{X}_1 + \mathfrak{X}_2\right)\otimes \mathfrak{X}_3 = \mathfrak{X}_1 \otimes \mathfrak{X}_3 + \mathfrak{X}_2 \otimes \mathfrak{X}_3, \\ & \mathfrak{X}_1 \otimes (\mathfrak{X}_2 + \mathfrak{X}_3) = \mathfrak{X}_1 \otimes \mathfrak{X}_2 + \mathfrak{X}_1 \otimes \mathfrak{X}_3.
\end{align}
Also, notice that
\begin{align}
\mathfrak{X}_i \otimes (\mathfrak{X}_j \otimes \mathfrak{X}_k) = (\mathfrak{X}_i \otimes \mathfrak{X}_j) \otimes \mathfrak{X}_k, \ \ \ \ \ \ \ \ i, j, k \in \{1, 2, 3\}.
\end{align}
Suppose that $T_{i,j}$ and $S_{i,j}$ are mappings from $\mathcal{A}$ into itself, where $i \in \{1, ..., m\}$ and $j \in \{1, ..., n\}$. It is easy to check that the following equation holds:
\begin{align}
\left(\sum_{j = 1}^{n}\prod_{i = 1}^{m}T_{i, j}(a_i)\right) \otimes \left(\sum_{j = 1}^{n}\prod_{i = 1}^{m}S_{i, j}(a_i)\right) = \sum_{r = 1}^{n} \sum_{j = 1}^{k}\left(\prod_{i = 1}^{m}T_{i, r}S_{i, j}(a_i)\right)
\end{align}

The identity element with respect to the operation $\otimes$ is $I(a)I(b) = ab$ which is denoted by $\textbf{1}_{(a,b)}^{\otimes}$, i.e. $\textbf{1}_{(a,b)}^{\otimes} = ab.$
\\
\\
Now we are going to obtain Leibniz's rule for $(\sigma, \tau$)-derivations using the same method as we calculated Leibniz's rule for derivation. Recall that the commutator of the mappings $S$ and $T$ is $[T, S] = TS - ST$. \\

Let $d:\mathcal{A} \rightarrow \mathcal{A}$ be a $(\sigma, \tau$)-derivation such that $[\tau, d] = [\sigma, d] =0$. Let $a$ and $b$ be two arbitrary elements of $\mathcal{A}$. We have
$$d(ab) = d(a) \sigma(b) + \tau(a) d(b)$$
Putting $\mathfrak{X} = d(a)\sigma(b)$ and $\mathfrak{Y} = \tau(a)d(b)$, we have $d(ab) = \mathfrak{X} + \mathfrak{Y}$.  In view of (2.1), we consider the operation $\otimes$ between $\mathfrak{X}$ and $\mathfrak{Y}$ as follows:
$$\mathfrak{X} \otimes \mathfrak{Y} = (d(a)\sigma(b)) \otimes (\tau(a)d(b)) = d \tau (a) \sigma d(b).$$
Since we are assuming that $d \sigma = \sigma d$ and $d \tau = \tau d$, it is observed that $\mathfrak{X} \otimes \mathfrak{Y} = \mathfrak{Y} \otimes \mathfrak{X}$ and also $$\mathfrak{X}^{n} = \underbrace{\mathfrak{X} \otimes ... \otimes \mathfrak{X}}_{n-times} = \underbrace{d(a)\sigma(b) \otimes ... \otimes d(a)\sigma(b)}_{n-times} = d^n(a)\sigma^n(b)$$ and $$\mathfrak{Y}^{n} = \underbrace{\mathfrak{Y} \otimes ... \otimes \mathfrak{Y}}_{n-times} = \tau^n(a)d^n(b)$$ for all $n \in \mathbb{N}$. According to Newton's binomial formula, we have $$(\mathfrak{X} + \mathfrak{Y})^n = \sum_{k = 0}^{n}\Big(_{k}^{n}\Big)\mathfrak{X}^{n - k} \otimes \mathfrak{Y}^{k}$$ for all $n \in \mathbb{N},$ where $\mathfrak{X}^{0} = \mathfrak{Y}^{0} = \textbf{1}_{(a,b)}^{\otimes} = ab$. Also note that
\begin{align*}
(\mathfrak{X} + \mathfrak{Y})^n & = \sum_{k = 0}^{n}\Big(_{k}^{n}\Big)\mathfrak{X}^{n - k} \otimes \mathfrak{Y}^{k} \\ & = \sum_{k = 0}^{n}\Big(_{k}^{n}\Big)\big(d^{n - k}(a)\sigma^{n - k}(b)\big) \otimes \big(\tau^{k}(a) d^{k}(b)\big) \\ & = \sum_{k = 0}^{n}\Big(_{k}^{n}\Big)d^{n - k}( \tau^{k}(a)) d^{k}(\sigma^{n - k}(b)),
\end{align*}
which means that
\begin{align}
(\mathfrak{X} + \mathfrak{Y})^n = \sum_{k = 0}^{n}\Big(_{k}^{n}\Big)d^{n - k}( \tau^{k}(a)) d^{k}(\sigma^{n - k}(b)).
\end{align}
We claim that $d^n(ab) = (\mathfrak{X} + \mathfrak{Y})^n$ for all $n \in \mathbb{N}$. We proceed the proof by induction. It is clear that the previous equality is true for $n = 1$. As induction assumption, suppose that our claim is true for the positive integer $n$. Hence, we have
\begin{align*}
d^{n + 1}(ab) & = d \big(d^{n}(ab)\big) = d \Big(\sum_{k = 0}^{n}\Big(_{k}^{n}\Big)\tau^{k}\big(d^{n - k}(a)\big)\sigma^{n -k}\big(d^{k}(b)\big)\Big) \\ & = 
\sum_{k = 0}^{n}\Big(_{k}^{n}\Big)\tau^{k}\left(d^{n + 1 - k}(a)\right)\sigma^{n + 1 - k}\left(d^{k}(b)\right) +  \\ & \sum_{k = 0}^{n}\Big(_{k}^{n}\Big)\tau^{k + 1}\left(d^{n - k}(a)\right)\sigma^{n - k}\left(d^{k + 1}(b)\right) \\ & = \sum_{k = 0}^{n}\Big(_{k}^{n}\Big)\tau^{k}\left(d^{n + 1 - k}(a)\right)\sigma^{n + 1 - k}\left(d^{k}(b)\right) + \\& \sum_{k = 1}^{n + 1}\Big(_{k - 1}^{n}\Big)\tau^{k}\left(d^{n + 1 - k}(a)\right)\sigma^{n + 1 - k}\left(d^{k}(b)\right) \\ & = \sum_{k = 1}^{n}\Big(_{k}^{n}\Big)\tau^{k}\left(d^{n + 1 - k}(a)\right)\sigma^{n + 1 - k}\left(d^{k}(b)\right) + \\ & \sum_{k = 1}^{n}\Big(_{k - 1}^{n}\Big)\tau^{k}\left(d^{n + 1 - k}(a)\right)\sigma^{n + 1 - k}\left(d^{k}(b)\right) + d^{n+ 1}(a)\sigma^{n + 1}(b) + \tau^{n + 1}(a) d^{n + 1}(b)\\ & = \sum_{k = 1}^{n}\left[\Big(_{k}^{n}\Big) + \Big(_{k - 1}^{n}\Big)\right]\tau^{k}\left(d^{n + 1 - k}(a)\right)\sigma^{n + 1 - k}\left(d^{k}(b)\right) + \\ & d^{n+ 1}(a)\sigma^{n + 1}(b) + \tau^{n + 1}(a) d^{n + 1}(b) \\ & = \sum_{k = 1}^{n}\Big(_{k}^{n + 1}\Big)\tau^{k}\left(d^{n + 1 - k}(a)\right)\sigma^{n + 1 - k}\left(d^{k}(b)\right)+ \\ & d^{n+ 1}(a)\sigma^{n + 1}(b) + \tau^{n + 1}(a) d^{n + 1}(b) \\ & = \sum_{k = 0}^{n + 1}\Big(_{k}^{n + 1}\Big)d^{n + 1 - k}\left(\tau^{k}(a)\right)d^k \left(\sigma^{n + 1 - k}(b)\right),
\end{align*}
which means that
\begin{align}
d^{n + 1}(ab) = \sum_{k = 0}^{n + 1}\Big(_{k}^{n + 1}\Big)d^{n + 1 - k}\left(\tau^{k}(a)\right)d^k \left(\sigma^{n + 1 - k}(b)\right).
\end{align}
Comparing (2.6) and (2.7), we deduce that
\begin{align}
d^{n + 1}(ab) = (\mathfrak{X} + \mathfrak{Y} )^{n + 1}.
\end{align}
Hence, we get that
\begin{align*}
d^n(ab) & = (\mathfrak{X} + \mathfrak{Y})^n \\ & = \sum_{k = 0}^{n}\Big(_{k}^{n}\Big)d^{n - k}\big(\tau^{k}(a)\big)d^{k}\big(\sigma^{n -k}(b)\big) \\ & = \sum_{k = 0}^{n}\Big(_{k}^{n}\Big)\tau^{k}\big(d^{n - k}(a)\big) \sigma^{n -k} \big(d^{k}(b)\big)
\end{align*}
for all $n \in \mathbb{N}$.\\

We can also get the above result more easily. It is clear that $d(ab) = \mathfrak{X} + \mathfrak{Y}$. As induction assumption, suppose that our claim is true for the positive integer $n$, i.e. $d^{n}(ab) = (\mathfrak{X} + \mathfrak{Y})^n = \sum_{k = 0}^{n}\Big(_{k}^{n}\Big)d^{n - k}( \tau^{k}(a)) d^{k}(\sigma^{n - k}(b)) $. Now, we have the following expressions:
\begin{align*}
d^{n + 1}(ab) & = d^{n}(ab) \otimes d(ab) \\ & = (\mathfrak{X} + \mathfrak{Y})^{n} \otimes (\mathfrak{X} + \mathfrak{Y}) \\ & = (\mathfrak{X} + \mathfrak{Y})^{n + 1} \\ & = \sum_{k = 0}^{n + 1}\Big(_{k}^{n + 1}\Big)\mathfrak{X}^{n + 1 - k} \otimes \mathfrak{Y}^{k} \\ & = \sum_{k = 0}^{n + 1}\Big(_{k}^{n + 1}\Big) \big(d^{n + 1 -k}(a) \sigma^{n + 1 - k}(b)\big) \otimes \big(\tau^{k}(a) d^{k}(b) \big) \\ & = \sum_{k = 0}^{n + 1}\Big(_{k}^{n + 1}\Big)d^{n + 1 - k}\left(\tau^{k}(a)\right)\sigma^{n + 1 - k}\left(d^k (b)\right) \\ & = \sum_{k = 0}^{n + 1}\Big(_{k}^{n + 1}\Big)d^{n + 1 - k}\left(\tau^{k}(a)\right)d^k \left(\sigma^{n + 1 - k}(b)\right),
\end{align*}
as desired.
\\
\\
In the following theorem, using the idea expressed for derivations and $(\sigma, \tau)$-derivations, we obtain Leibniz's rule for $(g_1, h_1, g_2, h_2)$-derivations.

\begin{theorem} \label{2.1} Let $\mathcal{A}$ be an algebra and let $f:\mathcal{A} \rightarrow \mathcal{A}$ be a $(g_1, h_1, g_2, h_2)$-derivation, that is, $f(ab) = g_1(a) h_1(b) + g_2(a) h_2(b)$ for all $a, b \in \mathcal{A}$.  Suppose that $[g_1, g_2] = [h_1, h_2] =0$. Then, for each $n \in \mathbb{N}$ and $a, b \in \mathcal{A}$, we have
$$f^{n}(ab) = \sum_{k = 0}^{n}\Big(_{k}^{n}\Big)g_1^{n - k}\big(g_2^{k}(a)\big)h_1^{n -k}\big(h_2^{k}(b)\big),$$
where $g_1^{0} = g_2^0 = h_1^0 = h_2^{0} = I$.
\end{theorem}
\begin{proof} Let $a$ and $b$ be two arbitrary elements of $\mathcal{A}$. We have
$$f(ab) = g_1(a) h_1(b) + g_2(a) h_2(b)$$
Putting $\mathfrak{X}_i = g_i(a)h_i(b)$, where $i \in \{1,2\}$, we have $f(ab) = \mathfrak{X}_1 + \mathfrak{X}_2$. In view of (2.1), we consider the operation $\otimes$ between $\mathfrak{X}_1$ and $\mathfrak{X}_2$ as follows: $$\mathfrak{X}_1 \otimes \mathfrak{X}_2 = (g_1(a)h_1(b)) \otimes (g_2(a)h_2(b)) = g_1 g_2 (a) h_1 h_2(b).$$
Since we are assuming that $[g_1,g_2] = [h_1, h_2] = 0$, it is observed that $\mathfrak{X}_1 \otimes \mathfrak{X}_2 = \mathfrak{X}_2 \otimes \mathfrak{X}_1$. Moreover, we have $$\mathfrak{X}_1^{n} = \underbrace{\mathfrak{X}_1 \otimes ... \otimes \mathfrak{X}_1}_{n-times} = \underbrace{g_1(a)h_1(b) \otimes ... \otimes g_1(a)h_1(b)}_{n-times} = g_1^n(a)h_1^n(b)$$ and $$\mathfrak{X}_2^{n} = \underbrace{\mathfrak{X}_2 \otimes ... \otimes \mathfrak{X}_2}_{n-times} = g_2^n(a)h_2^n(b)$$ for all $n \in \mathbb{N}$. Assume that $\mathfrak{X}_i^{0} = \textbf{1}_{(a,b)}^{\otimes} = ab$, where $i \in \{2, 3\}$.
Our task is to show that $f^{n}(ab) = (\mathfrak{X}_1 + \mathfrak{X}_2)^{n}$ for all $n \in \mathbb{N}$. We proceed the proof by induction. Obviously, the previous equality is true for $n = 1$. As induction assumption, suppose that our claim is true for the positive integer $n$, i.e. \begin{align*}
f^n(ab) & = (\mathfrak{X}_1 + \mathfrak{X}_2)^n \\ & = \sum_{k = 0}^{n}\Big(_{k}^{n}\Big)\mathfrak{X}_1^{n - k} \otimes \mathfrak{X}_2^{k} \\ & =
\sum_{k = 0}^{n}\Big(_{k}^{n}\Big)\big(g_1^{n - k}(a)h_1^{n - k}(b)\big) \otimes \big(g_2^{k}(a)h_2^{k}(b)\big) \\ & = \sum_{k = 0}^{n}\Big(_{k}^{n}\Big)g_1^{n - k}\big(g_2^{k}(a)\big)h_1^{n -k}\big(h_2^{k}(b)\big).
\end{align*}

We now have:
\begin{align*}
f^{n + 1}(ab) & = f^{n}(ab) \otimes f(ab) \\ & = (\mathfrak{X}_1 + \mathfrak{X}_2)^{n} \otimes (\mathfrak{X}_1 + \mathfrak{X}_2) \\ & = (\mathfrak{X}_1 + \mathfrak{X}_2)^{n + 1} \\ & = \sum_{k = 0}^{n + 1}\Big(_{k}^{n + 1}\Big)\mathfrak{X}_1^{n + 1 - k} \otimes \mathfrak{X}_2^{k} \\ & = \sum_{k = 0}^{n+1}\Big(_{k}^{n+1}\Big)\big(g_1^{n+1 - k}(a)h_1^{n+1 - k}(b)\big) \otimes \big(g_2^{k}(a)h_2^{k}(b)\big) \\ & =
\sum_{k = 0}^{n + 1}\Big(_{k}^{n +1}\Big)g_1^{n+ 1 - k}\big(g_2^{k}(a)\big)h_1^{n+ 1 -k}\big(h_2^{k}(b)\big).
\end{align*}
Thereby, our proof is complete.
\end{proof}

In the following, there are some immediate corollaries.

\begin{corollary} Let $\mathcal{A}$ be an algebra, let $\Delta:\mathcal{A} \rightarrow \mathcal{A}$ be an additive generalized $(\sigma, \tau$)-derivation corresponding to an additive mapping $d$ such that $[\tau, \Delta] = [\sigma, d] =0$. Then, for each $n \in \mathbb{N}$ and $a, b \in \mathcal{A}$, we have
$$\Delta^{n}(ab) = \sum_{k = 0}^{n}\Big(_{k}^{n}\Big)\Delta^{n - k}\big(\tau^{k}(a)\big)d^{k}\big(\sigma^{n -k}(b)\big),$$
where $\Delta^{0} = \sigma^{0} = \tau^0 = d^0 = I$.
\end{corollary}
\begin{proof} It is enough to note that $\Delta$ is, indeed, a $(\Delta, \sigma, \tau, d)$-derivation. Now, Theorem \ref{2.1} gives the result.
\end{proof}

\begin{corollary} \label{3} Let $\mathcal{A}$ be an algebra, let the triple $(d_1, d_2, d_3):\mathcal{A} \rightarrow \mathcal{A}$ be a ternary derivation, that is, $d_1(ab) = d_2(a) b + a d_3(b)$ holds for all $a, b \in \mathcal{A}$. Then, for each $n \in \mathbb{N}$ and $a, b \in \mathcal{A}$, we have
$$d_1^{n}(ab) = \sum_{k = 0}^{n}\Big(_{k}^{n}\Big)d_2^{n - k}(a)d_3^{k}(b),$$
where $d_2^{0} = d_3^{0} = I$.
\end{corollary}
\begin{proof} We can consider $d_1$ as a $(d_2, I, I, d_3)$-derivation and the required result is obtained easily from Theorem \ref{2.1}.
\end{proof}

\section{calculation of Leibniz's rule via non-commutative Newton's binomial formula}

In the previous part, using Newton's formula, we obtained Leibniz's rule for some generalized types of derivations. Using the commutative Newton's formula made us consider restrictive conditions for those mappings. For example, we obtained Leibniz's rule for $(g_1, h_1, g_2, h_2)$-derivations with the condition that $[g_1, g_2] = [h_1, h_2] =0$. Leibniz's rule for $(\sigma, \tau)$-derivations was also obtained with the condition that $[\tau, d] = [\sigma, d] =0$. As can be seen, these conditions are restrictive. Hence, in this section, we are going to derive Leibniz's rule for generalized types of derivations using the non-commutative version of Newton's formula and without considering those restrictive conditions.

The author has studied the non-commutative version of Newton's binomial formula (see \cite{H}). First, we state some prerequisites. A linear mapping $\delta: \mathcal{A} \rightarrow \mathcal{A}$ is called a generalized derivation if there exists a derivation $d:\mathcal{A} \rightarrow \mathcal{A}$ such that $\delta$ satisfies $\delta(ab) = \delta(a) b + a d(b)$ for all $a, b \in \mathcal{A}$. Let $b_1, b_2$ be two arbitrary elements of $\mathcal{A}$. A linear mapping $\delta_{b_1, b_2}:\mathcal{A} \rightarrow \mathcal{A}$ defined by $\delta_{b_1, b_2}(a) = b_1 a - a b_2$ is a generalized derivation associated with both the inner derivations $d_{b_1}$ and $d_{b_2}$. Indeed, we have $$\delta_{b_1, b_2}(ab) = \delta_{b_1, b_2}(a)b + a d_{b_2}(b) = d_{b_1}(a) b + a \delta_{b_1, b_2}(b)$$ for all $a, b \in \mathcal{A}$. We say that $\delta_{b_1, b_2}$ is an inner generalized derivation. As seen, a generalized derivation $\delta$ can satisfy $\delta(ab) = a \delta(b) + d(a) b$ for all $a, b \in \mathcal{A}$.

We begin with the following Lemma.
\begin{lemma}\cite[Lemma 1]{H}\label{1} Let $\mathcal{A}$ be an algebra, and let $b_1, b_2$ be two arbitrary elements of $\mathcal{A}$. Then
\begin{align}
\delta_{b_1, b_2}^{n}(a) = \sum_{k = 0}^{n}(-1)^{k}\Big(_{k}^{n}\Big)b_1^{n - k} a b_2^{k},
\end{align}
for every non-negative integer $n$ and every $a \in \mathcal{A}$.
\end{lemma}

\begin{theorem}\label{2} \cite[Theorem 1]{H} Let $\mathcal{A}$ be a unital algebra. Then for any $a, b \in \mathcal{A}$ and any non-negative integer $n$, the following equality holds true
\begin{eqnarray}\label{2.3}
a^{n} = \sum_{k = 0}^{n}\Big(_{k}^{n}\Big)\delta_{a, b}^{n - k}(\textbf{1})b^{k}.
\end{eqnarray}
\end{theorem}

As a consequence of the above theorem, the non-commutative Newton's binomial formula is obtained.

\begin{corollary} \label{3} \cite[Corollary 1]{H} Let $\mathcal{A}$ be a unital algebra and $a, b \in \mathcal{A}$. Then, \\
 (The non-commutative Newton's binomial formula)
\begin{eqnarray*}
(a + b)^n & =& \sum_{k = 0}^{n}\Big(_{k}^{n}\Big)\delta_{a + b, c}^{n - k}(\textbf{1})c^{k} \\ & =& \sum_{k = 0}^{n}\sum_{j = 0}^{n - k}\Big(_{k}^{n}\Big)\Big(_{j}^{n - k}\Big)(-1)^{j}(a + b)^{n - k - j}c^{j + k},
\end{eqnarray*}
where $c$ is an arbitrary element of $\mathcal{A}$ and $n \in \mathbb{N} \cup \{0\}$.
\end{corollary}

We are now ready to obtain Leibniz's rule for generalized types of derivations in the general case. Let $a,b$ be two arbitrary elements of $\mathcal{A}$. Recall that the identity element with respect to the operation $\otimes$ is $I(a)I(b) = ab$ which is denoted by $\textbf{1}_{(a,b)}^{\otimes}$, i.e. $\textbf{1}_{(a,b)}^{\otimes} = ab.$

\begin{theorem} \label{4} Let $\mathcal{A}$ be an algebra and let $f:\mathcal{A} \rightarrow \mathcal{A}$ be a $(g_1, h_1, g_2, h_2)$-derivation, that is, $f(ab) = g_1(a) h_1(b) + g_2(a) h_2(b)$ for all $a, b \in \mathcal{A}$. Set $\mathfrak{X}_1 = g_1(a) h_1(b)$ and $\mathfrak{X}_2 = g_2(a) h_2(b)$. Then, for each $n \in \mathbb{N}$ and $a, b \in \mathcal{A}$, we have
\begin{align}
f^{n}(ab) & = (\mathfrak{X}_1 + \mathfrak{X}_2)^{n} = \sum_{k = 0}^{n}\Big(_{k}^{n}\Big)\delta_{\mathfrak{X}_1 + \mathfrak{X}_2, \mathfrak{X}_2}^{n - k}(\textbf{1}_{(a,b)}^{\otimes}) \otimes \mathfrak{X}_2^{k} \\ & = \sum_{k = 0}^{n}\sum_{j = 0}^{n - k}\Big(_{k}^{n}\Big)\Big(_{j}^{n - k}\Big)(-1)^{j}(\mathfrak{X}_1 + \mathfrak{X}_2)^{n - k - j}\otimes \mathfrak{X}_2^{j + k},
\end{align}
where $\delta_{\mathfrak{X}_1 + \mathfrak{X}_2, \mathfrak{X}_2}^{0}= g_1^{0} = g_2^0 = h_1^0 = h_2^{0} = I$.
\end{theorem}
\begin{proof} Let $a$ and $b$ be two arbitrary elements of $\mathcal{A}$. We have
$$f(ab) = g_1(a) h_1(b) + g_2(a) h_2(b).$$
It is observed that $f(ab) = \mathfrak{X}_1 + \mathfrak{X}_2$. In view of (2.1), we consider the operation $\otimes$ between $\mathfrak{X}_1$ and $\mathfrak{X}_2$ as follows: $$\mathfrak{X}_1 \otimes \mathfrak{X}_2 = (g_1(a)h_1(b)) \otimes (g_2(a)h_2(b)) = g_1 g_2 (a) h_1 h_2(b).$$
Moreover, we have  $$\mathfrak{X}_1^{n} = \underbrace{\mathfrak{X}_1 \otimes ... \otimes \mathfrak{X}_1}_{n-times} = \underbrace{g_1(a)h_1(b) \otimes ... \otimes g_1(a)h_1(b)}_{n-times} = g_1^n(a)h_1^n(b)$$ and $$\mathfrak{X}_2^{n} = \underbrace{\mathfrak{X}_2 \otimes ... \otimes \mathfrak{X}_2}_{n-times} = g_2^n(a)h_2^n(b)$$ for all $n \in \mathbb{N}$. Also, $\mathfrak{X}_i^{0} = \textbf{1}_{(a,b)}^{\otimes}$, where $i \in \{2, 3\}$. According to Corollary \ref{3}, we have
\begin{align}
(\mathfrak{X}_1 + \mathfrak{X}_2)^{n} & = \sum_{k = 0}^{n}\Big(_{k}^{n}\Big)\delta_{\mathfrak{X}_1 + \mathfrak{X}_2, \mathfrak{X}_2}^{n - k}(\textbf{1}_{(a,b)}^{\otimes}) \otimes \mathfrak{X}_2^{k} \\ & = \sum_{k = 0}^{n}\sum_{j = 0}^{n - k}\Big(_{k}^{n}\Big)\Big(_{j}^{n - k}\Big)(-1)^{j}(\mathfrak{X}_1 + \mathfrak{X}_2)^{n - k - j}\otimes \mathfrak{X}_2^{j + k}.
\end{align}
Our task is to show that $f^{n}(ab) = (\mathfrak{X}_1 + \mathfrak{X}_2)^{n}$ for all $n \in \mathbb{N}$. We proceed the proof by induction. Obviously, the previous equality is true for $n = 1$. As induction assumption, suppose that our claim is true for the positive integer $n$, i.e.
$f^n(ab) = (\mathfrak{X}_1 + \mathfrak{X}_2)^n$.

We now have:
\begin{align*}
f^{n + 1}(ab) & = f^{n}(ab) \otimes f(ab) \\ & = (\mathfrak{X}_1 + \mathfrak{X}_2)^{n} \otimes (\mathfrak{X}_1 + \mathfrak{X}_2) \\ & = (\mathfrak{X}_1 + \mathfrak{X}_2)^{n + 1} \\ & = \sum_{k = 0}^{n + 1}\Big(_{k}^{n+1}\Big)\delta_{\mathfrak{X}_1 + \mathfrak{X}_2, \mathfrak{X}_2}^{n+1 - k}(\textbf{1}_{(a,b)}^{\otimes}) \otimes \mathfrak{X}_2^{k},
\end{align*}
as desired.
\end{proof}

\begin{example} Let $\mathcal{A}$ be an algebra and let $f:\mathcal{A} \rightarrow \mathcal{A}$ be a $(g_1, h_1, g_2, h_2)$-derivation. Using (3.3), we calculate $f^{2}$ and $f^{3}$. First, we calculate $f^{2}$, directly. Let $a, b$ be two arbitrary elements of $\mathcal{A}$. We have
\begin{align*}
f^2(ab) & = f \left(g_1(a) h_1(b) + g_2(a) h_2(b)\right) \\ & = f\left(g_1(a) h_1(b)\right) + f\left(g_2(a) h_2(b)\right) \\ & = g_1^{2}(a) h_1^{2}(b) + g_2 g_1(a) h_2 h_1(b) + g_1 g_2(a) h_1 h_2(b) + g_2^{2}(a) h_2^{2}(b),
\end{align*}
which means that
\begin{align*}
f^2(ab) =  g_1^{2}(a) h_1^{2}(b) + g_2 g_1(a) h_2 h_1(b) + g_1 g_2(a) h_1 h_2(b) + g_2^{2}(a) h_2^{2}(b),
\end{align*}
Now, using (3.3), we calculate $f^{2}(ab)$. In view of (2.4), we have $(\mathfrak{X}_1 + \mathfrak{X}_2) \otimes \mathfrak{X}_i = \mathfrak{X}_1 \otimes \mathfrak{X}_i + \mathfrak{X}_2 \otimes \mathfrak{X}_i$, where $i \in \{2, 3\}$. Hence, we have the following expressions:

\begin{align*}
f^2(ab) & =  (\mathfrak{X}_1 + \mathfrak{X}_2)^{2} = \sum_{k = 0}^{2}\Big(_{k}^{2}\Big)\delta_{\mathfrak{X}_1 + \mathfrak{X}_2, \mathfrak{X}_2}^{2 - k}(\textbf{1}_{(a,b)}^{\otimes}) \otimes \mathfrak{X}_2^{k} \\ & = \delta_{\mathfrak{X}_1 + \mathfrak{X}_2, \mathfrak{X}_2}^{2 }(\textbf{1}_{(a,b)}^{\otimes}) \otimes \mathfrak{X}_2^{0} + 2 \delta_{\mathfrak{X}_1 + \mathfrak{X}_2, \mathfrak{X}_2}(\textbf{1}_{(a,b)}^{\otimes}) \otimes \mathfrak{X}_2^{1} + \delta_{\mathfrak{X}_1 + \mathfrak{X}_2, \mathfrak{X}_2}^{0}(\textbf{1}_{(a,b)}^{\otimes}) \otimes \mathfrak{X}_2^{2} \\ & = \delta_{\mathfrak{X}_1 + \mathfrak{X}_2, \mathfrak{X}_2}\left((\mathfrak{X}_1 + \mathfrak{X}_2) \otimes \textbf{1}_{(a,b)}^{\otimes} - \textbf{1}_{(a,b)}^{\otimes} \otimes \mathfrak{X}_2 \right) + \\ & 2 \left((\mathfrak{X}_1 + \mathfrak{X}_2) \otimes \textbf{1}_{(a,b)}^{\otimes} - \textbf{1}_{(a,b)}^{\otimes} \otimes \mathfrak{X}_2 \right) \otimes \mathfrak{X}_2 + \textbf{1}_{(a,b)}^{\otimes} \otimes \mathfrak{X}_2^{2} \\ & = (\mathfrak{X}_1 + \mathfrak{X}_2) \otimes \mathfrak{X}_1 - \mathfrak{X}_1 \otimes \mathfrak{X}_2 + 2 \mathfrak{X}_1 \otimes \mathfrak{X}_2 + \mathfrak{X}_2^{2} \\ & = \mathfrak{X}_1^{2} + \mathfrak{X}_1 \otimes \mathfrak{X}_2 + \mathfrak{X}_2 \otimes \mathfrak{X}_1 + \mathfrak{X}_2^{2} \\ & = g_1^{2}(a) h_1^{2}(b) + g_1 g_2(a) h_1 h_2(b) + g_2 g_1(a) h_2 h_1(b)  + g_2^{2}(a) h_2^{2}(b)
\end{align*}

It is observed that
\begin{align*}
f^{2}(ab) = (\mathfrak{X}_1 + \mathfrak{X}_2)^{2}.
\end{align*}
Using (3.3), the interested reader can calculate $f^{3}(ab)$ and also $f^{n}(ab)$ for any $n \in \mathbb{N}$.
\end{example}

\begin{remark} Using Theorem \ref{4}, we can obtain Leibniz's rule for $(\sigma, \tau$)-derivations, generalized $(\sigma, \tau$)-derivations, ternary derivations and so on without assuming any restrictive condition.
\end{remark}


\bigskip
\noindent

\end{document}